\documentclass[10pt]{amsart}
\usepackage{graphicx}
\usepackage{mathpazo}
\usepackage{amsrefs}
\usepackage[colorlinks]{hyperref}

\newtheorem{theorem}{Theorem}[section]

\newtheorem{lemma}[theorem]{Lemma}

\theoremstyle{definition}

\newtheorem{remark}[theorem]{Remark}
\newtheorem{example}[theorem]{Example}

\numberwithin{equation}{section}

\begin{document}
\title[]{Region crossing change on nonorientable surfaces}
\author{Zhiyun Cheng}
\address{School of Mathematical Sciences, Laboratory of Mathematics and Complex Systems, MOE, Beijing Normal University, Beijing, 100875, China}
\email{czy@bnu.edu.cn}
\author{Jingze Song}
\address{School of Mathematical Sciences, Beijing Normal University, Beijing 100875, China}
\email{jzsong@mail.bnu.edu.cn}

\begin{abstract}
In this paper, we give a classification of link diagrams on nonorientable surfaces up to region crossing changes.
\end{abstract}
\maketitle

\section{Introduction}\label{section1}
An unknotting operation is a local operation on a knot diagram such that any knot diagram can be transformed into a knot diagram representing the unknot via finitely many such operations. One of the most well-studied unknotting operation is crossing change, which switches the overstrand and understrand of a crossing point. The reader is referred to \cite{Sch1998} for a nice expository survey about the role of crossing change in knot theory. Besides of crossing change, there are several other unknotting operations. For example, in \cite{Mur1985} Murakami studied the $\#$-operation and showed that it is an unknotting operation. Later in \cite{MN1989}, Murakami and Nakanishi proved that the $\triangle$-operation is an unknotting operation. The $n$-gon move, which switches all the crossing points alternatingly appear on the boundary of an $n$-gon, was proved to be an unknotting operation by Aida in \cite{Aid1992}. 

\begin{figure}[h]
\centering
\includegraphics{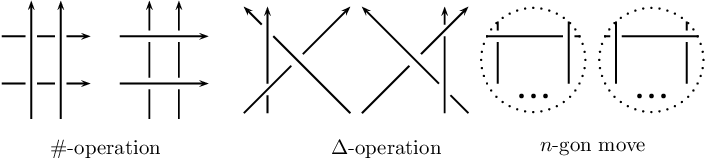}\\
\caption{Some unknotting operations}\label{figure1}
\end{figure}

Recently, a new local operation, called \emph{region crossing change}, was introduced by Shimizu in \cite{Shi2014}. Roughly speaking, for a given knot diagram, applying region crossing change on a region $R$ switches all the crossing points on the boundary of $R$. Figure \ref{figure2} provides an example to show that how to obtain a diagram of the unknot from a diagram of the trefoil knot by applying region crossing change on the region $R$. Let $L$ be a link diagram and $C(L)$ the set of crossing points, for a subset $P\subseteq C(L)$, we say $P$ is \emph{region crossing change admissible} if there exist some regions such that applying region crossing changes on these regions switch all the crossing points in $P$ and preserve all the crossing points in $C(L)\setminus P$. The main result proved by Shimizu in \cite{Shi2014} is that any crossing point in a knot diagram is region crossing change admissible. As a corollary, it follows that region crossing change is an unknotting operation for knot diagrams.

\begin{figure}[h]
\centering
\includegraphics{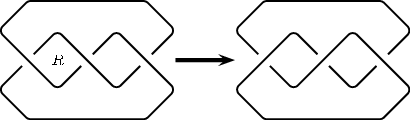}\\
\caption{Region crossing change on $R$}\label{figure2}
\end{figure}

\begin{remark}
The careful reader will have spotted that both the $\#$-operation and the $n$-gon move are special cases of region crossing change. In fact, if one wants to use $\#$-operation to unknot a knot diagram, one needs not only the $\#$-operation but also Reidemeister moves. So are the $\triangle$-operation and the $n$-gon move. For region crossing change, Reidemeister moves are prohibited. In other words, during the process of region crossing changes, the shadow (the underlying 4-valent plane graph, also called the link projection or the medial graph) of the link diagram is actually fixed.
\end{remark}

The research on region crossing change has made considerable progress over the past decade. Shortly after Shimizu proved that region crossing change is an unknotting operation, the first named author proved that region crossing change is an unknotting operation on a link diagram if and only if this link is proper \cite{CG2012, Che2013}. Recall that a link is called \emph{proper} if for any of its components, the sum of the linking numbers between this component and other components is an even integer. It turns out that whether region crossing change is an unknotting operation does not depend on the choice of the link diagram. Shimizu's result is equivalent to the fact that the incidence matrix (see Section \ref{section2} for the definition) is of full rank over $\mathbb{Z}_2$. It immediately follows that it is of full rank over $\mathbb{Z}$, therefore the integral region crossing change is also an unknotting operation on knot diagrams \cite{AS2012}. In \cite{HSS2015, ST2020}, the effect of region crossing change on spatial graph diagrams is studied. As a complement of region crossing change, the region freeze crossing change, which preserves the crossing points on the boundary of a region and switches all others, was also proved to be an unknotting operation on knot diagrams \cite{IS2016}. Very recently, the operation of region crossing change was extended to origami by Oshikiri, Shimizu and Tamura in \cite{OST2024}.

The original region crossing change is an local operation on link diagrams on the plane, or equivalently $S^2$. It is a natural question to consider the region crossing change on surfaces of arbitrary high genus. By using Bollobás-Riordan-Tutte polynomial, this question was first addressed by Dasbach and Russell in \cite{DR2018} with some restricted conditions. Later, a complete classification of link diagrams on $\Sigma_g$ modulo region crossing change was given in \cite{CCXZ2022}, where $\Sigma_g$ denotes a closed orientable surface of genus $g$. Currently, the only case that has not been conquered is the classification of link diagrams on nonorientable surfaces. More precisely, given a link diagram $L$ on $N_g$ with $c$ crossing points, where $N_g$ denotes the connected sum of $g$ real projective plane, by switching these crossing points one obtains totally $2^c$ different link diagrams. What we want to know is, how many equivalence classes do these link diagrams have under the equivalence relation induced by the region crossing change? It turns out that this number is equal to $2^{c-\text{rank}(M_L)}$. Here $M_L\in M_{r\times c}(\mathbb{Z}_2)$ denotes the incidence matrix of $L$, and $r$ is the number of regions of $N_g\setminus L$. The main result of this paper is the following.

\begin{theorem}\label{Theorem1}
Let $L=K_1\cup\cdots\cup K_n$ be a link diagram on $N_g$, then $\text{rank}(M_L)=r-n-1+\text{rank}(N_L)$.
\end{theorem}

Here $r$ denotes the number of regions and $N_L\in M_{n\times g}(\mathbb{Z}_2)$ is the homology matrix of $L$ (see Section \ref{section2} for the definition).

The rest of this paper is arranged as follows. In Section \ref{section2} we recall the definitions of incidence matrix $M_L$ and homology matrix $N_L$, and explain the roles they play in the investigation of region crossing change. Section \ref{section3} is devoted to give a proof of Theorem \ref{Theorem1}. Later, for a given subset $P$ of the set of crossing points we discuss when the crossing points in $P$ are region crossing change admissible. Together with the result of Theorem \ref{Theorem1}, these provide a complete understanding to the effect of region crossing change on surfaces.

\section{Incidence matrix and homology matrix}\label{section2}
\subsection{Region crossing change}
Let $F$ be a connected closed surface, not necessary orientable, and $L$ a link diagram on $F$. Choose a region $R$ from $F\setminus L$, taking \emph{region crossing change} on $R$ means to switch each crossing point on the boundary of $R$ $k$ times if $R$ appears $k$ times around this crossing point. For example, Figure \ref{figure3} depicts a 6-component link diagram on the torus. Consider the five crossing points $c_k$ $(0\leq k \leq 4)$, then $c_k$ will be switched $k$ times if one applies region crossing change on the region $R$, since $R$ appears $k$ times around the crossing point $c_k$ $(0\leq k \leq 4)$.

\begin{figure}[h]
\centering
\includegraphics{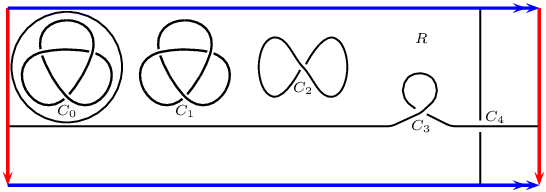}\\
\caption{A 6-component link diagram on $T^2$}\label{figure3}
\end{figure}

\begin{remark}
In the original definition of region crossing change introduced by Shimizu in \cite{Shi2014}, each crossing point on $\partial R$ will be switched exactly once if one applies region crossing change on $R$, no matter how many times $R$ appear around it. For example, if a knot diagram on the plane contains a nugatory crossing point, then it will be changed if one applies region crossing change on the unbounded region. In \cite{AS2012}, this is called \emph{single counting rule}. The adoption of the counting rule used in the present paper corresponds to the \emph{double counting rule} in \cite{AS2012}.
\end{remark}

\begin{remark}
It should be noted that if we adopt the double counting rule, just as what we choose in this paper, the result that region crossing change is an unknotting operation for knot diagrams can be obtained from the original definition of Alexander polynomial introduced in \cite{Ale1928}, together with the fact that knots have odd determinants.
\end{remark}

\subsection{Incidence matrix and Homology matrix}
Given a link diagram $L=K_1\cup\cdots\cup K_n$ on a connected closed surface $F$, let us use $c_1, \cdots, c_c$ and $R_1, \cdots, R_r$ to denote the crossing points and regions, respectively. The \emph{incidence matrix} of $L$, which is denoted by $M_L$, is defined as $M_L=(a_{ij})_{r\times c}\in M_{r\times c}(\mathbb{Z}_2)$, where $a_{ij}$ is equal to 1 if $R_i$ appears odd times around $c_j$ and 0 otherwise. The incidence matrix precisely demonstrates the effect of region crossing change on link diagrams. For example, if we want to understand the effect of applying region crossing changes on $R_i$ $(i\in P)$, where $P$ is a subset of $\{1, \cdots, r\}$, we just need to consider the sum $\sum\limits_{i\in P}r_i\in M_{1\times c}(\mathbb{Z}_2)$. Here $r_i$ denotes the row vector of $M_L$ corresponding to $R_i$. As another example, the fact that region crossing change is an unknotting operation for knot diagrams on the plane is equivalent to the fact that the incidence matrix of any knot diagram is of full rank.

For the homology matrix, let us first assume that the surface $F$ is orientable, say $F=\Sigma_g=gT^2$. Denote a basis of $H_1(\Sigma_g; \mathbb{Z}_2)=\bigoplus\limits_{2g}\mathbb{Z}_2$ by $a_1, \cdots, a_{2g}$ and the homology class of each component $K_i$ $(1\leq i\leq n)$ by $[K_i]$. If $[K_i]=\sum\limits_{j=1}^{2g}b_{ij}a_j$, then we define the \emph{homology matrix} of $L$ to be $N_L=(b_{ij})_{n\times 2g}\in M_{n\times 2g}(\mathbb{Z}_2)$. If $F=N_g=g\mathbb{R}P^2$, which is nonorientable, we choose a basis of $H_1(g\mathbb{R}P^2; \mathbb{Z}_2)=\bigoplus\limits_g\mathbb{Z}_2$, say $a_1, \cdots, a_g$. Then the \emph{homology matrix} is defined as $N_L=(b_{ij})_{n\times g}\in M_{n\times g}(\mathbb{Z}_2)$ if $[K_i]=\sum\limits_{j=1}^gb_{ij}a_j$. Obviously, the rank of homology matrix $N_L$ does not depend on the choice of the basis.

\begin{example}
Consider the two link diagrams illustrated in Figure \ref{figure4}. On the left side we have a 3-component link diagram $L_1$ on $T^2$ and the right side depicts a 2-component link diagram $L_2$ on $2\mathbb{R}P^2$. Since the information of over/undercrossing of each crossing point is of no importance, we omit it here. Now we have
\begin{center}
$M_{L_1}=\begin{pmatrix}
1&1&0&1&1&0\\
1&1&0&0&0&1\\
1&1&0&1&1&0\\
1&1&1&0&0&0\\
0&0&1&1&1&0\\
0&0&0&1&1&1
\end{pmatrix},
N_{L_1}=\begin{pmatrix}
1&0\\
1&0\\
1&0
\end{pmatrix}$,\\
$M_{L_2}=\begin{pmatrix}
1&1&0&1&1&0\\
1&1&0&0&0&1\\
1&0&1&1&1&0\\
1&1&1&0&0&0\\
0&0&1&1&1&0\\
0&1&1&1&1&1
\end{pmatrix},
N_{L_2}=\begin{pmatrix}
1&0\\
0&0
\end{pmatrix}.$
\end{center}
Direct calculation shows that rank$(M_{L_1})=3$, rank$(N_{L_1})=1$, rank$(M_{L_2})=4$, and rank$(N_{L_2})=1$. For $i\in\{1, 2\}$, we have rank$(M_{L_i})=r_i-n_i-1+\text{rank}(N_{L_i})$. Here $r_i$ and $n_i$ denote the number of regions and the number of components of $L_i$. 
\begin{figure}
\centering
\includegraphics{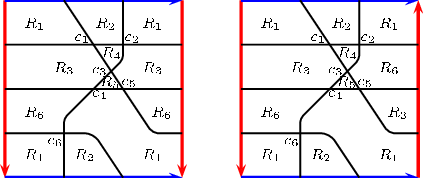}\\
\caption{A 3-component link diagram on $T^2$ and a 2-component link diagram on Klein bottle}\label{figure4}
\end{figure}
\end{example}

Given a link diagram $L$ with $c$ crossing points, fix an orientation of each component of $L$ and an order of all crossing points. We assign an element $0\in\mathbb{Z}_2$ to each positive crossing point and $1\in\mathbb{Z}_2$ to each negative crossing point. Now the link diagram $L$ corresponds to a row vector $r_L\in M_{1\times c}(\mathbb{Z}_2)$. It is easy to observe that applying region crossing change on $R_i$ turns $L$ into a new link diagram corresponds to $r_L+r_i$. Here $r_i$ denotes the row vector corresponds to $R_i$. Thus two link diagrams $L$ and $L'$ are related by some region crossing changes if and only if there exists a subset $P\subseteq\{1, \cdots, r\}$ such that $r_{L'}=r_L+\sum\limits_{i\in P}r_i$. It follows immediately that, among all the $2^c$ link diagrams sharing the same shadow with $L$, the number of equivalent classes of link diagrams modulo region crossing changes is equal to $2^{c-\text{rank}M_L}$. 

For link diagrams on orientable surfaces, we have the following result.

\begin{theorem}\cite{CCXZ2022}\label{Theorem2}
Let $L=K_1\cup\cdots\cup K_n$ be a link diagram on $\Sigma_g$, then 
\begin{center}
$\text{rank}(M_L)=r-n-1+\text{rank}(N_L)$.
\end{center}
Here $r$ denotes the number of regions.
\end{theorem}

If $K$ is a knot diagram on the plane, then $n=1, r=c+2$ and rank$(N_K)=0$. Hence the rank of $M_K$ equals $c$, which recovers the result of Shimizu with respect to the double counting rule.

\begin{remark}
The number of equivalence classes of all the link diagrams on $\Sigma_g$ sharing the same shadow was first studied by Dasbach and Russell in \cite{DR2018}. Under the assumption that the link diagram is cellularly embedded which admits a checkerboard fashion coloring, Dasbach and Russell provided a formula to calculate the number of equivalence classes (hence the rank of $M_L$) via the Tait graph. The reader is referred to \cite{DR2018} for more details.
\end{remark}

\section{The proof of the main theorem}\label{section3}
The main idea of the proof of Theorem \ref{Theorem1} is to convert a nonorientable surface into an orientable one, and then use the result of Theorem \ref{Theorem2}. Before giving the proof of Theorem \ref{Theorem1}, we need the following lemma.

\begin{lemma}\label{lemma3}
Let $L$ be a link diagram on $N_g$ with $r$ regions, then $r-\text{rank}(M_L)$ is invariant under the second Reidemeister move.
\end{lemma}

\begin{figure}
\centering
\includegraphics{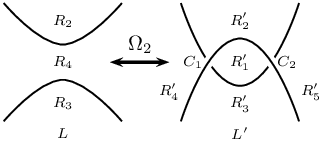}\\
\caption{Reidemeister move $\Omega_2$}\label{figure5}
\end{figure}

\begin{proof}
The proof is essentially the same to that of the orientable case given in \cite[Theorem 3.4]{CCXZ2022}. We sketch it here.

Let us use $L'$ to denote the new link diagram obtained from $L$ by a second Reidemeister move, and use $c_1, c_2$ to denote the two new-born crossing points. For the newly emerging bigon, we use $R_1'$ to denote it. See Figure \ref{figure5}. Assume $R_2', R_3', R_4', R_5'$ indicate distinct four regions of $L'$, then $L'$ has two more regions than $L$, it is sufficient to show that $\text{rank}(M_L)+2=\text{rank}(M_{L'})$.

One computes
\begin{center}
$\text{rank}(M_{L'})=\text{rank}\begin{pmatrix}
1&1&0&\cdots&0\\
1&1&&r_2'&\\
1&1&&r_3'&\\
1&0&&r_4'&\\
0&1&&r_5'&\\
\vdots&\vdots&&&
\end{pmatrix}=\text{rank}\begin{pmatrix}
1&1&0&\cdots&0\\
0&0&&r_2'&\\
0&0&&r_3'&\\
0&0&&r_4'+r_5'&\\
0&1&&r_5'&\\
\vdots&\vdots&&&
\end{pmatrix}=\text{rank}\begin{pmatrix}
1&1&0&\cdots&0\\
0&1&&r_5'&\\
0&0&&&\\
0&0&&M_L&\\
0&0&&&\\
\vdots&\vdots&&&
\end{pmatrix}=\text{rank}(M_L)+2$.
\end{center}

It is worthy to point out that it is possible some of $\{R_2', R_3', R_4', R_5'\}$ indicate actually the same region of $L'$. It suffices to combine the corresponding row vectors into one row vector, the proof above still holds.
\end{proof}

Now we turn to the proof of Theorem \ref{Theorem1}.

\begin{proof}
Suppose we are given a link diagram $L$ on $N_g$. Without loss of generality, let us assume that $g$ is even. If not, we can make a cross-cap addition to $N_g$ such that the excised disk does not intersect with the link diagram $L$. Note that the addition of an $\mathbb{R}P^2$ does not affect all the terms appearing in the equation in the statement of Theorem \ref{Theorem1}. Thus, from now on $L$ is considered as a link diagram on $N_{2g}$. Let us choose a polygonal presentation of $N_{2g}$ as below
\begin{center}
$\Pi=a_1b_1a_1^{-1}b_1^{-1}\cdots a_{g-1}b_{g-1}a_{g-1}^{-1}b_{g-1}^{-1}a_gb_ga_gb_g^{-1}$.
\end{center}
In other words, here we regard the surface $N_{2g}$ as $(g-1)T^2\#K$, where $K$ denotes the Klein bottle. 

Assume there are $m$ arcs of $L$ having nonempty intersection with the edge $a_g$ between $b_g$ and $b_g^{-1}$. After taking a sequence of the second Reidemeister moves on these arcs we obtain a new link diagram $L'$ with $2m-2$ new crossing points. Let us use $c_1, \cdots, c_{m-1}$ to denote the lower $m-1$ crossing points, and $R_1, \cdots, R_m$ to denote the regions incident to $a_g$, see Figure \ref{figure6}. This purpose is to make the next half twist operation only affects the regions near $a_g$. According to Lemma \ref{lemma3}, $r-\text{rank}(M_L)$ is preserved under the second Reidemeister move. On the other hand, it is obvious that the number of components and rank$(N_L)$ are also preserved. Therefore it suffices to consider the link diagram $L'$.

Now we take a half twist on these $m$ arcs, which reverses the direction of $a_g$ (and also the positions of $R_2,\cdots, R_m$). It follows that now the polygonal presentation has the form
\begin{center}
$\Pi'=a_1b_1a_1^{-1}b_1^{-1}\cdots a_{g-1}b_{g-1}a_{g-1}^{-1}b_{g-1}^{-1}a_gb_ga_g^{-1}b_g^{-1}$,
\end{center}
which represents $gT^2$. The new link diagram $L''$ has $\frac{m^2-m}{2}$ more crossing points and $\frac{m^2-m}{2}$ more regions than that of $L'$. Assume $L'$ has $c$ crossing points and $r$ regions, let us use $c_{c+1}, \cdots, c_{c+\frac{m^2-m}{2}}$ and $R_{r+1}, \cdots, R_{r+\frac{m^2-m}{2}}$ to denote the new crossing points and new regions created by the half twist, see Figure \ref{figure6}. Note that the half twist preserves the number of components and the rank of homology matrix.

\begin{figure}[h]
\centering
\includegraphics{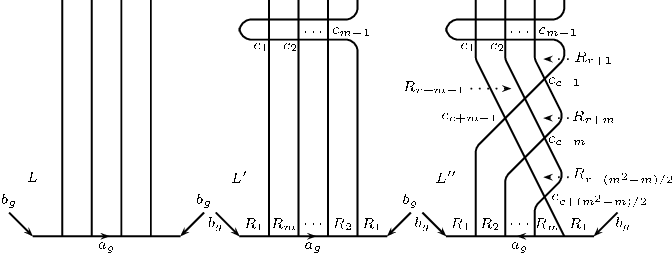}\\
\caption{Transformation from $L$ to $L''$}\label{figure6}
\end{figure}

As a link diagram on $gT^2$, according to Theorem \ref{Theorem2}, for $L''$ we have
\begin{center}
$\text{rank}(M_{L''})=r+\frac{m^2-m}{2}-n-1+\text{rank}(N_{L''})$.
\end{center}
Since $\text{rank}(N_{L'})=\text{rank}(N_{L''})$, in order to prove Theorem \ref{Theorem1}, it is sufficient to show that  
\begin{center}
$\text{rank}(M_{L''})=\text{rank}(M_{L'})+\frac{m^2-m}{2}$.
\end{center}

Let us divide the set of regions of $L'$ into $\{R_1\}\cup\{R_2, \cdots, R_m\}\cup\{R_{m+1}, \cdots, R_r\}$, and divide the set of crossing points of $L'$ into $\{c_1, \cdots, c_{m-1}\}\cup\{c_m, \cdots, c_c\}$, then the incidence matrix $M_{L'}$ has the block form
\begin{center}
$M_{L'}=\begin{pmatrix}
M_{11}&M_{12}\\
M_{21}&M_{22}\\
M_{31}&M_{32}
\end{pmatrix}$,
\end{center}
where 
\begin{center}
$M_{11}=\begin{pmatrix}
1&0&\cdots&0&1
\end{pmatrix}$, and
$M_{21}=\begin{pmatrix}
0&0&0&\cdots&0&0&1\\
0&0&0&\cdots&0&1&1\\
\vdots&\vdots&\vdots&&\vdots&\vdots&\vdots\\
0&1&1&\cdots&0&0&0\\
1&1&0&\cdots&0&0&0
\end{pmatrix}$.
\end{center}
With the partition of the regions of $L''$ as 
\begin{center}
$\{R_1\}\cup\{R_2, \cdots, R_m\}\cup\{R_{m+1}, \cdots, R_r\}\cup\{R_{r+1}, \cdots, R_{r+\frac{m^2-m}{2}}\}$,
\end{center}
and the partition of the crossing points of $L''$ as 
\begin{center}
$\{c_1, \cdots, c_{m-1}\}\cup\{c_m, \cdots, c_c\}\cup\{c_{c+1}, \cdots, c_{c+\frac{m^2-m}{2}}\}$,
\end{center}
the incidence matrix of $L''$ has the following block from
\begin{center}
$M_{L''}=\begin{pmatrix}
M_{11}&M_{12}&M_{13}\\
\textbf{0}&M_{22}&M_{23}\\
M_{31}&M_{32}&\textbf{0}\\
M_{41}&\textbf{0}&M_{43}
\end{pmatrix}$.
\end{center}
Here
\begin{center}
$M_{13}=\begin{pmatrix}
1&\textbf{0}_{1\times(m-3)}&1&1&\textbf{0}_{1\times(m-3)}&1&\textbf{0}_{1\times(m-2)}&\cdots&0&1
\end{pmatrix}$,
$M_{23}=\begin{pmatrix}
\textbf{0}_{1\times(m-2)}&1&\textbf{0}_{1\times(m-3)}&1&\textbf{0}_{1\times\frac{m^2-5m+6}{2}}\\
\textbf{0}_{1\times(2m-4)}&1&\textbf{0}_{1\times(m-4)}&1&\textbf{0}_{1\times\frac{m^2-7m+12}{2}}\\
\vdots&\vdots&\vdots&\vdots&\vdots\\
0&0&\cdots&0&1
\end{pmatrix}$,
$M_{41}=\begin{pmatrix}
0&0&0&\cdots&0&0&1\\
0&0&0&\cdots&0&1&1\\
\vdots&\vdots&\vdots&&\vdots&\vdots&\vdots\\
0&1&1&\cdots&0&0&0\\
1&1&0&\cdots&0&0&0\\
0&0&0&\cdots&0&0&0\\
\vdots&\vdots&\vdots&\vdots&\vdots&\vdots&\vdots\\
0&0&0&\cdots&0&0&0
\end{pmatrix}=
\begin{pmatrix}
M_{21}\\
\textbf{0}_{\frac{m^2-3m+2}{2}\times(m-1)}
\end{pmatrix}$.
\end{center}
The matrix $M_{43}$ is a bit more complicated. In order to describe it, let us set
\begin{center}
$P_{s\times s}=\begin{pmatrix}
1&0&0&\cdots&0&0&0\\
1&1&0&\cdots&0&0&0\\
\vdots&\vdots&\vdots&&\vdots&\vdots&\vdots\\
0&0&0&\cdots&1&1&0\\
0&0&0&\cdots&0&1&1
\end{pmatrix}$ and
$Q_{(s-1)\times s}=\begin{pmatrix}
1&1&0&\cdots&0&0\\
0&1&1&\cdots&0&0\\
\vdots&\vdots&\ddots&\ddots&\vdots&\vdots\\
0&0&\cdots&1&1&0\\
0&0&\cdots&0&1&1
\end{pmatrix}$.
\end{center}
Then $M_{43}$ can be written as
\begin{center}
$M_{43}=\begin{pmatrix}
P_{(m-1)\times(m-1)}&\textbf{0}&\cdots&\textbf{0}&\textbf{0}\\
Q_{(m-2)\times(m-1)}&P_{(m-2)\times(m-2)}&\cdots&\textbf{0}&\textbf{0}\\
\vdots&\ddots&\ddots&\vdots&\vdots\\
\textbf{0}&\cdots&Q_{2\times3}&P_{2\times2}&\textbf{0}\\
\textbf{0}&\textbf{0}&\cdots&Q_{1\times2}&P_{1\times1}
\end{pmatrix}$.
\end{center}
Notice that the matrix $M_{43}$ is a lower triangular matrix with all the main diagonal elements being 1, thus we can transform $M_{43}$ into an identity matrix $M_{43}'=I_{\frac{m^2-m}{2}\times\frac{m^2-m}{2}}$ by a sequence of elementary column transformations. The key point here is, after these column transformations, the matrices $M_{13}$ and $M_{23}$ are transformed into
\begin{center}
$M_{13}'=\begin{pmatrix}
\textbf{0}_{1\times(m-1)}&\textbf{1}_{1\times\frac{m^2-3m+2}{2}}
\end{pmatrix}$, and 
$M_{23}'=\begin{pmatrix}
I_{(m-1)\times(m-1)}&X_{(m-1)\times\frac{m^2-3m+2}{2}}
\end{pmatrix}$.
\end{center}
Now let us use the first $m-1$ columns of $M_{43}'$, which is an identity matrix, to eliminate the nonzero entries in $M_{41}$. Recall that the first $m-1$ rows of $M_{41}$ is exactly the matrix $M_{21}$ and the first $m-1$ columns of $M_{23}'$ is also an identity matrix. It immediately follows that when $M_{41}$ becomes a zero matrix, the zero sub-matrix in the second row and first column of the block form of $M_{L''}$ will be transformed into $M_{21}$. Since the first $m-1$ entries in $M_{13}'$ are all zeros, these operations preserve the block $M_{11}$ in $M_{L''}$. Since now $M_{41}$ is a zero matrix, we can use the identity matrix $M_{43}'$ to eliminate all the nonzero entries in $M_{13}'$ and $M_{23}'$. Finally, we find that
\begin{center}
$\text{rank}(M_{L''})=\text{rank}\begin{pmatrix}
M_{11}&M_{12}&\textbf{0}\\
M_{21}&M_{22}&\textbf{0}\\
M_{31}&M_{32}&\textbf{0}\\
\textbf{0}&\textbf{0}&I_{\frac{m^2-m}{2}\times\frac{m^2-m}{2}}
\end{pmatrix}=\text{rank}\begin{pmatrix}
M_{L'}&\textbf{0}\\
\textbf{0}&I_{\frac{m^2-m}{2}\times\frac{m^2-m}{2}}
\end{pmatrix}$,
\end{center}
which implies that
\begin{center}
$\text{rank}(M_{L''})=\text{rank}(M_{L'})+\frac{m^2-m}{2}$.
\end{center}
The proof is finished.
\end{proof}

\section{Which crossing points are region crossing change admissible?}\label{section4}
Theorem \ref{Theorem1} and Theorem \ref{Theorem2} tell us the number of equivalence classes of all the link diagrams on a connected closed surface sharing the same shadow, modulo region crossing changes. In order to complete the classification, we need to figure out that for an arbitrary selection of crossing points, whether these chosen crossing points are region crossing change admissible. For link diagrams on the plane, it was first noticed by the first-named author \cite{Che2013} that if $L=K_1\cup\cdots\cup K_n$ is an $n$-component link diagram then for any $c_1\in K_{i_1}\cap K_{i_2}, c_2\in K_{i_2}\cap K_{i_3}, \cdots, c_{p-1}\in K_{i_{p-1}}\cap K_{i_p}$ and $c_p\in K_{i_p}\cap K_{i_1}$, then $c_1, \cdots, c_p$ are region crossing change admissible. Conversely, if some crossing points are region crossing change admissible, then they can be written as the union of finitely many such ``cyclically-chosen" crossing points. Recently, Kengo Kawamura gave another diagrammatic characterization for the subsets of crossing points that are region crossing change admissible \cite{Kaw2021}. In this subsection, we mainly follow the idea of \cite[Theorem 3.2.1]{Kaw2021} and extend it to link diagrams on closed surfaces.

\begin{remark}
Note that determining whether some crossing points are region crossing change admissible is equivalent to find a basis for the range space of $M_L^T$. Here $M_L^T$ denotes the transpose of $M_L$. A closely related question is to find ineffective sets of regions with respect to region crossing change, which is equivalent to find a basis for the null space of $M_L^T$. For a given link diagram, a set of regions is called \emph{ineffective} if region crossing changes on these regions change no crossing point. For example, with respect to the double counting rule, for a knot diagram with a checkerboard coloring, the set of all black regions and the set of all white regions are both ineffective. Actually, these two sets are the only two nontrivial ineffective sets of regions. The reader is referred to \cite{IS2016} and \cite{CMR2020} for a complete characterization of ineffective sets of regions.
\end{remark}

Consider a link diagram $L$ on a connected closed surface $F$, which may be orientable or nonorientable. By a \emph{semi-arc} of $L$, we mean a line segment of $L$ between two adjacent crossing points. Fix a set $P$ of crossing points, following \cite{Kaw2021}, a \emph{bi-coloring} of $(L, P)$ is a map $\phi: SC(L)\to\mathbb{Z}_2$ which assigns an element of $\mathbb{Z}_2$ to each semi-arc such that for each crossing point $c\in P$, each pair of opposite semi-arcs receive distinct colors, and for each crossing point $c\notin P$, each pair of opposite semi-arcs receive the same color. Here $SC(L)$ denotes the set of semi-arcs of $L$. See Figure \ref{figure7}.

\begin{figure}[h]
\centering
\includegraphics{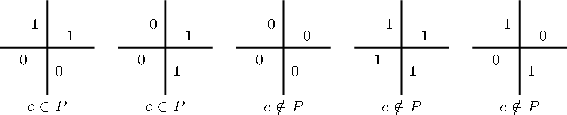}\\
\caption{Bicoloring at a crossing point $c$}\label{figure7}
\end{figure}

For a given bi-coloring $\phi$ of $(L, P)$, consider all the semi-arcs colored by 1, which is the union of several immersed closed curves on $F$. Let us denote the corresponding homology class of it by $[\phi]\in H_1(F; \mathbb{Z}_2)$. Then we have the following result.

\begin{theorem}
Let $L$ be a link diagram on a connected closed surface $F$, $P\subseteq C(L)$ a subset of the set of crossing points, then $P$ is region crossing change admissible if and only if there exists a bi-coloring $\phi: SC(L)\to\mathbb{Z}_2$ such that $[\phi]=0\in H_1(F;\mathbb{Z}_2)$. 
\end{theorem}
\begin{proof}
If there exists a bi-coloring $\phi$ and the union of all the semi-arcs colored by 1 give rise to a null homologous element in $H_1(F;\mathbb{Z}_2)$. Then there exists a compact subsurface $S\subset F$ (possibly disconnected) whose boundary $\partial S$ coincides with the union of all 1-colored semi-arcs. According to the definition of the bi-coloring, it is easy to observe that applying region crossing changes on the regions in $S$ exactly switches the crossing points in $P$. Therefore, $P$ is region crossing change admissible.

Conversely, if $P$ is region crossing change admissible, then there exist a family of regions such that region crossing changes on them exactly change the crossing points in $P$. Now let us color these regions black and color the remaining regions white. For a given semi-arc, if the two regions adjacent to it have the same color then we assign 0 to it, otherwise we assign 1 to it. It is not difficult to find that this provides a bi-coloring $\phi$ of $(L, P)$. Since the boundary of black regions consists of all the semi-arcs colored by 1, it follows that $[\phi]=0\in H_1(F;\mathbb{Z}_2)$.
\end{proof}

\section*{Acknowledgement}
Zhiyun Cheng was supported by the NSFC grant 12371065. Jingze Song was supported by a REU program at Beijing Normal University.

\bibliographystyle{amsplain}
\bibliography{}

\end{document}